\documentclass[11pt]{article}
\usepackage{amssymb}
\usepackage{graphicx}
\usepackage{xcolor} 
\usepackage{tensor}
\usepackage{fullpage} 
\usepackage{amsmath}
\usepackage{amsthm}
\usepackage{verbatim}

\usepackage{enumitem}
\setlist[enumerate]{leftmargin=1.5em}
\setlist[itemize]{leftmargin=1.5em}

\usepackage{hyperref}
\usepackage{seqsplit,mathtools}

\definecolor{green}{rgb}{0,0.8,0} 

\newtheorem{theorem}{Theorem}[section]

\newtheorem{lemma}[theorem]{Lemma}

\theoremstyle{definition}

\theoremstyle{remark}
\newtheorem{remark}[theorem]{Remark}

\numberwithin{equation}{section}
\newcommand{\nrm}[1]{\Vert#1\Vert}

\newcommand{\nnrm}[1]{{\vert\kern-0.25ex\vert\kern-0.25ex\vert #1 
		\vert\kern-0.25ex\vert\kern-0.25ex\vert}}

\newcommand{\alp}{\alpha}

\newcommand{\eps}{\epsilon}

\newcommand{\omg}{\omega}

\newcommand{\bbN}{\mathbb N}

\newcommand{\bbR}{\mathbb R}

\newcommand{\calA}{\mathcal A}

\newcommand{\sd}{\triangle}

\begin{document}
\bibliographystyle{plain}
 \title{On the winding number for particle trajectories\\ in a disk-like vortex patch of   the Euler equations 
 }
\author{Kyudong Choi\thanks{Department of Mathematical Sciences, Ulsan National Institute of Science and Technology. Email: kchoi@unist.ac.kr} 
	\and In-Jee Jeong\thanks{Department of Mathematical Sciences, Seoul National University. E-mail: injee\_j@snu.ac.kr}
}

\date\today

\maketitle

\renewcommand{\thefootnote}{\fnsymbol{footnote}}

\footnotetext{\emph{Key words:} 2D Euler, vortex patch, stability, winding number,  large time behavior,  particle trajectory. \quad\emph{2010 AMS Mathematics Subject Classification:} 76B47, 35Q35 }

\renewcommand{\thefootnote}{\arabic{footnote}}

\begin{abstract} We consider vortex patch solutions of the incompressible Euler equations in the plane. It is shown that the winding number around the origin for most particles in the patch grows linearly in time when the initial patch is close to a disk enough.
\end{abstract}\vspace{1cm} 
{\Large \section{Introduction}}
\noindent We consider the  incompressible 
Euler equation in vorticity form in 
$\mathbb{R}^2$:
 \begin{equation}\begin{split}\label{main_eq}
\partial_t\omega
 + u \cdot \nabla \omega
&= 
0\quad\mbox{for }  x\in\mathbb{R}^2\mbox{ and  for }  t> 0,\\
 \omega|_{t=0}&=\omega_0 \quad\mbox{for }  x\in\mathbb{R}^2,
\end{split}\end{equation} where the Biot-Savart law is given by
$u=(K*\omega)$ with
 \begin{equation*} 
K(x):=\frac{1}{2\pi}\frac{x^\perp}{|x|^2}=\frac{1}{2\pi}\Big(-\frac{x_2}{|x|^2},\frac{x_1}{|x|^2}\Big).\end{equation*}
When $\omega_0$ lies on $L^1\cap L^\infty$, the existence and uniqueness of a global-in-time weak solution is due to Yudovich \cite{yud}. We are concerned with vortex patch solutions; when $\omega_0={\mathbf{1}}_{\Omega_0}$ for some bounded measurable set $\Omega_0\subset \mathbb{R}^2$, 
the corresponding Yudovich solution $\omega$ has the form of $\omega(t)={\mathbf{1}}_{\Omega_t}$ where 
$\Omega_t=\{
\phi_x(t) \in\mathbb{R}^2\,|\, x\in\Omega_0
\}$. Here $\phi_x(\cdot)$ is the unique solution to the system $$\frac{d}{dt}\phi_x(t)=u(t,\phi_x(t))\quad\mbox{for }t>0\quad\mbox{and } \quad \phi_x(0)=x,$$ which is well-defined thanks to
 the   estimate \begin{equation}\label{bddspeed_gen}
\begin{split}
\|u(t)\|_{L^\infty} \le C \nrm{\omega_0}_{L^1\cap L^\infty}, \quad t\geq 0 
\end{split}
\end{equation} and
 the log-Lipschitz estimate \begin{equation}\label{loglip_gen}
\begin{split}
|u(t,x)-u(t,x')| \le C \nrm{\omega_0}_{L^1\cap L^\infty} |x-x'|\ln(\frac{1}{|x-x'|}), \quad |x-x'|\le \frac{1}{2}, \quad t\geq 0. 
\end{split}
\end{equation}{The above estimates  in turn guarantee that for each $T>0$, there exists constants $c=c(T), C=C(T)>0$ satisfying
\begin{equation}\label{eq:flow-estimate}
\begin{split}
c|x-x'|^{ \exp(C\nrm{\omega_0}_{L^1\cap L^\infty}t ) } \le 	|\phi_{x}(t)-\phi_{x'}(t)| \le C|x-x'|^{ \exp(- C\nrm{\omega_0}_{L^1\cap L^\infty}t ) },\quad x,x'\in\mathbb{R}^2,  \quad t\in[0,T]
\end{split}
\end{equation}}  (\textit{e.g.} see  \cite[Sec. 8.2.3]{MB} or \cite[Ch. 2]{MaPu}).  \\

In this paper, we produce an estimate on the winding number around the origin for fluid particle trajectories when the initial patch is  disk-like.  Here we say that a patch on $\Omega$ is disk-like if the measure of the   symmetric difference $$\Omega\bigtriangleup B_r:=(\Omega\backslash B_r)\cup(B_r\backslash\Omega)$$ is small enough for some $r>0$ where we denote $B_r:=\{x\in\mathbb{R}^2\,|\,|x|<r\}$.  
 {Under the assumption that  $x\in\bbR^2\backslash\{0\}$  is a point whose trajectory never hits the origin (i.e. $\phi_{x}(t) \ne 0$ for all $t\ge 0$)\footnote{In Lemma \ref{lem:measure-zero} we shall prove that the set of points in $\bbR^2$ whose trajectory hits the origin after some time is of zero measure. }, we define the notion of winding number of $x$   up to time $t>0$ is defined by the integral \begin{equation*} 
		\begin{split}
N_x(t) := \frac{1}{2\pi} \int_0^t \frac{u_{tan}(s,\phi_x(s)) }{|\phi_x(s)|} ds,
\end{split} 
\end{equation*}
where  $u_{tan}$   is the speed in the angular direction with respect to the origin (see Definition \eqref{defn_tan}).}
  It is not difficult to check that this definition coincides with the usual notion of winding number (for curves in $\bbR^2$ not intersecting the origin) applied to the curve $\gamma_x:[0,t] \mapsto \bbR\backslash\{0\} $ defined by $\gamma_x(s)=\phi_x(s)$. For example, we observe that for the circular vortex patch supported on the unit disc $D:=B_1$, we have, for any $t>0$, \begin{equation}\label{disk_ex}
\begin{split}
{ \frac{N_x(t)}{t} = \frac{1}{4\pi} } ,\quad \forall x\in D\backslash\{ 0 \},
\end{split}
\end{equation} simply because $u_{tan}(t,x) = \frac{1}{2}|x|$ for   $t\geq0$ and $x\in D\backslash\{ 0 \}$. \\

 {There are several difficulties in treating the winding number. First, this quantity may decrease in time (particles can ``unwind'') and second, there is no uniform in space bound for the ratio $u(t,x)/|x|$. Rather, even assuming that the velocity vanishes at the origin ($u(t,0)=0$), taking $x'=0$ in \eqref{loglip_gen} shows that the quantity may diverge like $C\log(1/|x|)$ as $|x|\rightarrow 0$, which is indeed sharp for vortex patch solutions (\textit{e.g.} see \cite{BaCh}). } \\

\subsection{Main result} \ \\
We consider   the initial vorticity   given by the patch on some bounded measurable set $\Omega_0$ of unit strength, whose solution is identified with moving domains $\Omega_t$.  
We shall assume that $\Omega_0$ is a small perturbation of the unit disc $D$. 
We are now ready to state our main result on winding number for disk-like patches.  

\begin{theorem}\label{thm_infinite}

  For any $R>0$,  
  there exist  constants $\delta_0\in(0,1)$ and $C>0$   such that if
$$\Omega_0\subset B_{R}\quad\mbox{and}\quad  |D\triangle\Omega_0  |<\delta_0,$$  
then, by denoting
$\delta:=|D\triangle\Omega_0  |$,
  for any $T>0$, there exists a set 
 $H_{T}\subset \Omega_0$ such that
$$|H_{T}|\geq |\Omega_0|-C{\delta^{1/12}}
$$ and 
\begin{equation*}
{\left|\frac{N_x(T)}{T}  -  \frac{1}{4\pi}\right|\le } C {\delta^{1/12}}, \quad x\in H_T. \end{equation*}   
In addition, there exists a set 
 $H\subset \Omega_0$ such that
  $$|H|\geq |\Omega_0|-C{\delta^{1/12}}
$$ and
\begin{equation*}
{\liminf_{t\to\infty}\left| \frac{N_x(t)}{t}-\frac{1}{4\pi}\right|\le }C{\delta^{1/12}},\quad x\in H. 
\end{equation*}

\end{theorem}


 
 {Recalling \eqref{disk_ex}, the above result states that most particles wind around the origin roughly the same number of times as particles in the case of the disc. This can be interpreted as a more refined, dynamical in nature, notion of stability; the classical stability results (Wan and Pulvirenti \cite{wp} and Sideris and Vega \cite{sv}) simply say that for initially disc-like patches, the shape of the patch remains close to the disc for all times. 
	
Moreover, control on the winding number could be useful to prove certain \textit{instability} results for vortex patches and more generally for solutions to the 2D Euler equations. The basic idea is as follows: if two fluid particles rotate around the origin with different angular velocities while roughly keeping their respective distance from the origin, any curve connecting the particles gets stretched linearly in time. This observation was used in \cite{EJSVP2} to obtain spiral formation and \cite{CJ} for perimeter growth. 
}

\subsubsection*{Notations}
We collect the notations here that are used throughout the paper. \begin{itemize}
	\item 
Given $r>0$ and $x\in\mathbb{R}^2$, we define $B_r(x) :=\{ y\in \bbR^2\, |\, |y-x|<r \}$ and $B_r:=B_r(0)$. 	
	We denote the unit disc by $D := B_1$. The complement is denoted by $D^C$.
	\item Given $\eps>0$ and $i\in\bbN\cup\{0\}$, we set $B_i^\eps := B_{2^{-i}\eps}$ and $A_i^\eps := B_{i}^\eps\backslash B_{i+1}^{\eps}$. 
	\item The radial and tangential part of the velocity around the origin are defined by 
	\begin{equation}\label{defn_tan}
	u =u_{rad}\frac{x}{|x|}+u_{tan}\frac{x^\perp}{|x|}.\end{equation} Here we denote $x^\perp = (-x_2,x_1)$ for $x = (x_1,x_2)$. 
\end{itemize}

\section{Proof}
\begin{proof}[Proof of Theorem \ref{thm_infinite}]
Fix $R>0$ and let $\omega(t)={\mathbf{1}}_{\Omega_t}$ be the solution of \eqref{main_eq} with initial data $\omega_0={\mathbf{1}}_{\Omega_0} $ where $\Omega_0\subset B_{R}$. Take any 	  $\delta_0>0$ satisfying   $  \delta_0^{1/6}\leq 1/2$ and assume 
$$\delta:=|\Omega_0 \sd D |<\delta_0.$$  
If $\delta=0$, then we can take $H_T, \,H=D\setminus\{0\}$ by \eqref{disk_ex}. From now on, we assume $\delta>0$. \\

Throughout the paper, $C$ denotes a positive constant which may depend on $R>0$, may change from line to line, but it is independent of $\delta_0,\delta>0$.\\ 

\subsection{$L^1$-stability of a disk patch} \ \\
We recall the $L^1$-stability result for the circular vortex patch from \cite{sv}:
\begin{lemma}[{{\cite[Theorem 3]{sv}}}]
For any bounded open set $\Omega_0\subset \mathbb{R}^2$ and for any $r>0$, we have 
\begin{equation*}
\|{\mathbf{1}}_{\Omega_t}-{\mathbf{1}}_{B_r}\|_{L^1}^2\leq4\pi\cdot \sup_{x\in\Omega_0\bigtriangleup B_r}||x|^2-r^2|\cdot
\|{\mathbf{1}}_{\Omega_0}-{\mathbf{1}}_{B_r}\|_{L^1} \quad\mbox{for any }t> 0.
\end{equation*}
\end{lemma} 

\noindent

Note that $\nrm{\mathbf{1}_A-\mathbf{1}_B}_{L^1}=\int_{A\sd B} 1dx =|A\sd B|$.  Applying the above lemma with our $\Omega_0\subset B_R$ and $\mathbf{1}_{D}$, we obtain that \begin{equation}\label{defn_C1}
\begin{split}
|\Omega_t \sd D |\le C \sqrt{\delta},
 \quad   t>0.
\end{split}
\end{equation}
From this $L^1$ bound, we obtain that (\textit{e.g.} see \cite{Choi2019,EJSVP2}). \begin{lemma}\label{lem:vel-infty}
	Under the above assumptions on $\Omega_0$, we have with $u(t) = K* \mathbf{1}_{\Omega(t)}$ and $u_{D} = K * \mathbf{1}_D$ that \begin{equation*}
	\begin{split}
	\nrm{u(t) - u_{D}}_{L^\infty} \le C \delta^{1/4} 
	\end{split}
	\end{equation*} for all $t\ge0$. 
	 In particular, we have \begin{equation}\label{est_good_tan}
	\begin{split}
	\left|\frac{u_{tan}(t,x)}{|x|} - \frac{1}{2}\right| \le 
C \frac{\delta^{1/4}}{|x|},	
\quad x\in D\setminus\{0\}.
	\end{split}
	\end{equation}
\end{lemma} 
\begin{proof}
From the   estimate $|K(x)|\leq C/|x|$, we have
 \begin{equation*}
 \|K*f\|_{L^\infty}\leq C\|f\|^{1/2}_{L^1}\|f\|^{1/2}_{L^\infty}
 \end{equation*} for any $f\in (L^1\cap L^\infty)(\mathbb{R}^2)$  (\textit{e.g.}  see Lemma 2.1. in \cite{isg}). Thus, we can estimate, for any $x\in\mathbb{R}^2$,  
$$	|u(t,x)-u_{D}(x)|\leq 
\|K*({\mathbf{1}}_{\Omega_t}-{\mathbf{1}}_{D})\|_{L^\infty}\leq C\|{\mathbf{1}}_{\Omega_t}-{\mathbf{1}}_{D}\|^{1/2}_{L^1}\|{\mathbf{1}}_{\Omega_t}-{\mathbf{1}}_{D}\|^{1/2}_{L^\infty}\leq C|\Omega_t \sd D |^{1/2}.$$ Thus, by \eqref{defn_C1}, we get the first estimate.
Taking the tangential components and using that $u_{D,tan} = |x|/2$ for $x\in D\setminus\{0\}$, we obtain  \eqref{est_good_tan}.
\end{proof}

Next, we 
estimate time spent near the origin for particles. A similar estimate
can be found in  \cite[Lemma 2.2]{Choi2019}. 
\begin{lemma}\label{lem:bound} Under the above assumptions on $\Omega_0$,
	for any $T>0$, we have  \begin{equation*}
	\begin{split}
	\int_{\Omega_0} \left(  \frac{1}{T}\int_0^T \mathbf{1}_{ B_r}(\phi_x(t)) dt \right) dx \le C r^2,\quad r>0
	\end{split}
	\end{equation*} and \begin{equation*}
	\begin{split}
		\int_{\Omega_0} \left(  \frac{1}{T}\int_0^T \mathbf{1}_{D^C }(\phi_x(t)) dt \right) dx \le 
		 C \sqrt{\delta }.
	\end{split}
	\end{equation*} 
\end{lemma}
\begin{proof}
	We compute, using that the flow map $x\mapsto \phi_x(t)$ is area-preserving, for $r>0$ and $t>0$, \begin{equation*}
	\begin{split}
	\int_{\Omega_0} \mathbf{1}_{ B_{r} }(\phi_x(t))dx & = \int_{ \Omega_t \cap B_r} 1\,dx \le |B_r|=\pi r^2, 
	\end{split}
	\end{equation*} and,  by \eqref{defn_C1}, \begin{equation*}
	\begin{split}
		\int_{\Omega_0} \mathbf{1}_{D^C }(\phi_x(t))dx & = \int_{ \Omega_t \cap D^C} 1\,dx \le  \int_{ \Omega_t \sd D} 1\,dx = |\Omega_t \sd D |\leq C\sqrt{\delta} . 
	\end{split}
	\end{equation*} Now the desired bounds follow from integrating in time and applying Fubini's theorem. 
\end{proof}

\subsection{Set of particles touching the origin} \ \\
We prove the following result:
\begin{lemma}\label{lem:measure-zero}
	Let $\phi_x(t)$ be the flow map associated with a Yudovich solution. Then, we have {$$| \{ x \in \bbR^2 \, | \,  \phi_x(t) = 0 \mbox{ for some } t \ge 0  \} |= 0.$$}
\end{lemma}

This result is not trivial since for Yudovich solutions (in particular vortex patches) since the flow map is in general only H\"older continuous with H\"older exponent decaying to zero as $t\to\infty$. On the other hand, there are space-filling curves in the plane which is actually H\"older continuous with exponent $1/2$ (Hilbert's curve).

\begin{proof}[Proof of Lemma \ref{lem:measure-zero}]
	First step is to prove that the inverse of the flow map is H\"older continuous \textit{with respect to time}. For any given $t\ge0$, we consider $\phi(t,x)=\phi_{x}(t)$ as a bijection $\bbR^2\rightarrow\bbR^2$, and denote the inverse map by $A(t,\cdot)$; that is, $A(t,\phi(t,x))=x$. For this purpose, take some $t, t' \ge 0$ and observe that \begin{equation*}
		\begin{split}
			\phi(t,x)-\phi(t',x) = \int_{t'}^{t} u(\tau,\phi(\tau,x)) \, d\tau 
		\end{split}
	\end{equation*} gives \begin{equation*}
	\begin{split}
		\frac{|\phi(t,x)-\phi(t',x)|  }{|t-t'|} \le C \nrm{ u }_{L^\infty_{t,x}} \le C \nrm{\omega_0}_{L^1\cap L^\infty}. 
	\end{split}
\end{equation*}
 Then, for  $0\leq t'\leq t\leq 1$ and for $x\in\mathbb{R}^2$, we first take $a, a'$ such that $\phi(t',a')= \phi(t,a)=x$, so that \begin{equation*}
\begin{split}
	A(t,x)-A(t',x) = a - a'. 
\end{split}
\end{equation*} We have \begin{equation*}
\begin{split}
	\phi(t',a') - \phi(t,a') = \phi(t,a) - \phi(t,a'),
\end{split}
\end{equation*} and the left hand side is bounded in absolute value by $C\nrm{\omega_0}_{L^1\cap L^\infty}|t-t'|$. Using this together with the lower bound given in \eqref{eq:flow-estimate}, we deduce that 
there exist absolute constants $c, C>0$ such that
\begin{equation*}
\begin{split}
	c|a-a'|^{ \exp( Ct\nrm{\omega_0}_{L^1\cap L^\infty} ) } \le C\nrm{\omega_0}_{L^1\cap L^\infty}|t-t'|.
\end{split}
\end{equation*} 
 Hence we obtain \begin{equation}\label{eq:Holder-time}
\begin{split}
	|A(t,x)-A(t',x)| \le C( \nrm{\omega_0}_{L^1\cap L^\infty}|t-t'|)^{ \exp(- Ct\nrm{\omega_0}_{L^1\cap L^\infty} ) }, \quad x\in\mathbb{R}^2,\quad 0\leq t'\leq t\leq 1.
\end{split}
\end{equation} 

Now we fix any $\alpha\in(1/2, 1)$. Given the estimate \eqref{eq:Holder-time}, we may find constants $T_0\in(0,1)$ and $C_0>0$ depending only on $\nrm{\omega_0}_{L^1\cap L^\infty}$ such that \begin{equation}\label{eq:Holder-time2}
	\begin{split}
		|A(t,x)-A(t',x)| \le C_0|t-t'|^{\alp}, \quad x\in\mathbb{R}^2,\quad 0\leq t'\leq t\leq T_0.
	\end{split}
\end{equation}  We can now prove that {$$| \{ x \in \bbR^2 \, | \,  \phi_x(t) = 0 \mbox{ for some } t \in [0,T_0)  \} |= 0.$$} To see this, observe that the set of points $x$ satisfying $\phi_x(t)=0$ for some $t \in [0,T_0)$ is simply given by $\{ A(t,0) \, |\, t \in [0,T_0) \}=:\calA([0,T_0))$. Fix any integer $n\ge 1$ and consider the points $x_{k} := A(\frac{kT_0}{n},0)$ for $k=0,1,\dots,(n-1)$. From \eqref{eq:Holder-time2}, we see that \begin{equation*}
\begin{split}
	\calA([0,T_0)) \subset \bigcup_{ k = 0}^{n-1} B_{C_0 (\frac{T_0}{n})^\alp }(x_{k}),
\end{split}
\end{equation*}  which implies that \begin{equation*}
\begin{split}
	|\calA([0,T_0))| \le Cn^{1-2\alp} T_0^{2\alp}.
\end{split}
\end{equation*} In the last estimate, $C, T_0$ are independent of $n\geq 1$, and therefore we conclude that $|\calA([0,T_0))| =0$ thanks to  $\alpha>1/2$. 

\medskip

We now write $\phi(t;s,a)$ for the location of the particle at time $t$ which was $a$ at time $s$. For each integer $m\ge 0$, we consider the set \begin{equation*}
	\begin{split}
		U_{m} = \{ \phi(mT_0; s, 0) \,|\, s \in [mT_0, (m+1)T_0) \}. 
	\end{split}
\end{equation*} Note that $\calA([0,T_0))=U_{0}$. Applying the same argument (with initial data $\omg_0$ replaced with $\omega(t=mT_0)$), we deduce that $|U_{m}|=0$. This is possible only because $T_0$ depends only on $\nrm{\omg}_{L^1\cap L^\infty}$, which is independent of time. Next, since $A(t,\cdot)$ is measure preserving for any $t\ge0$, we see that $|A(mT_0,U_{m})|=0$. The set $A(mT_0,U_{m})$ is precisely the collection of points which touches the origin at some time $[mT_0,(m+1)T_0)$. This
finishes the proof. \end{proof}

\subsection{Estimate of winding number} \ \\
	Denote
	$$\epsilon:=\delta^{1/6}>0.$$
	Observe   $ \epsilon=\delta^{1/6}<\delta_0^{1/6}\leq 1/2$.
	Fix   $T>0$ and define \begin{equation*}
	\begin{split}
	G_i^T(x) := \frac{1}{T} \int_0^T \mathbf{1}_{ B_i^{\eps} }(\phi_x(t)) dt, \quad i\ge 0
	\end{split}
	\end{equation*} (recall the notation $B_i^\eps := B_{2^{-i}\eps}$) and \begin{equation*}
	\begin{split}
	G_{-1}^T(x) := \frac{1}{T} \int_0^T \mathbf{1}_{ D^C }(\phi_x(t)) dt. 
	\end{split}
	\end{equation*} Lemma \ref{lem:bound} states that 
	\begin{equation*}
	\begin{split}
	\int_{\Omega_0} G^T_i(x)dx \le C(2^{-i}\eps)^2,\quad   i\ge0,
	\end{split}
	\end{equation*} and, from Chebyshev's inequality, we obtain \begin{equation*}
	\begin{split}
	\left| \{ x \in \Omega_0\, | \,G^T_i(x) \ge{ (2^{-i}\eps)^{\frac{3}{2}} } \} \right|\le C{(2^{-i}\eps)^{\frac{1}{2}}}
	= C {2^{-i/2}\delta^{1/12}}, \quad   i\ge0. 
	\end{split}
	\end{equation*} Similarly,
we have 
	  \begin{equation*}
	\begin{split}
	\left| \{ x \in \Omega_0 \,|\, G^T_{-1}(x) \ge 2 \eps  \} \right|\le   C \frac{\sqrt{\delta }}{ \eps}=C\delta^{1/3}\leq C{\delta^{1/12}}.
	\end{split}
	\end{equation*}
	
	Now we define \begin{equation*}
	\begin{split}
	{ H_T :=\{ x \in \Omega_0\, |\, G^T_i(x) < (2^{-i}\eps)^{\frac{3}{2}}, \,\forall i\ge 0 \mbox{ and } G_{-1}^T(x)<2\eps \mbox{ and }    \phi_x(t) \ne 0,\, \forall t\in[0,T] \}. }
	\end{split}
	\end{equation*} 
	{Recalling Lemma \ref{lem:measure-zero},}  we have 	\begin{equation*}
	\begin{split}
	|H_T|\ge |\Omega_0|- C{\delta^{1/12}}. 
	\end{split}
	\end{equation*} For any  fixed $x\in H_T$, we split the time integral as follows: \begin{equation*}
	\begin{split}
	\int_0^T \frac{u_{tan}(t,\phi_x(t))}{|\phi_x(t)|} dt & = \int_{[0,T]\backslash (\cup_{i\ge-1}I_i) } \frac{u_{tan}(t,\phi_x(t))}{|\phi_x(t)|} dt+ \sum_{i \ge -1} \int_{I_i} \frac{u_{tan}(t,\phi_x(t))}{|\phi_x(t)|} dt,
	\end{split}
	\end{equation*} where \begin{equation*}
	\begin{split}
	I_i := \{ t \in [0,T]\, |\, \phi_x(t) \in A_i^\eps \},\quad i\ge 0
	\end{split}
	\end{equation*}(recall the notation $ A_i^\eps := B_i^\eps\backslash B_{i+1}^\eps$)  and \begin{equation*}
	\begin{split}
	I_{-1} := \{ t \in [0,T] \,|\, \phi_x(t)\in D^c \}.
	\end{split}
	\end{equation*} For   $i\ge 0 $, we have that $|I_i| \le TG_i^T(x)\leq {T(2^{-i}\eps)^{\frac{3}{2}}  }$ and {$|I_{-1}|\le CT\eps $}. Thus we get
$$|(\cup_{i\ge-1}I_i)|\leq CT\eps.$$

	For $i\ge0$, we also estimate, {  using the uniform bound $\nrm{u(t)}_{L^\infty} \le C(1+|\Omega_0|)\leq C$ which follows from \eqref{bddspeed_gen}, }\begin{equation*}
	\begin{split}
	\left| \int_{I_i} \frac{u_{tan}(t,\phi_x(t))}{|\phi_x(t)|} dt \right|
	\le  {C\int_{I_i} \frac{1}{|\phi_x(t)|} dt 
	 \le C |I_i| 2^{i} \eps^{-1} \le CT \eps^{1/2} 2^{-i/2}.}
	\end{split}
	\end{equation*}  In the case $i = -1$, we simply use $|\phi_x(t)|\geq 1$ together with $\nrm{u(t)}_{L^\infty} \le C$ to obtain  \begin{equation*}
	\begin{split}
	\left| \int_{I_{-1}} \frac{u_{tan}(t,\phi_x(t))}{|\phi_x(t)|} dt \right| \le C|I_{-1}| \le C T\eps.  
	\end{split}
	\end{equation*} Hence  we obtain \begin{equation*}
	\begin{split}
	\left| \sum_{i \ge -1} \int_{I_i} \frac{u_{tan}(t,\phi_x(t))}{|\phi_x(t)|} dt  \right| 
\le C T {\sqrt{\eps}}. 
	\end{split}
	\end{equation*} Note that the above constant $C>0$ is independent of $x\in H_T$. \\

	Lastly, for $t\in [0,T]\backslash (\cup_{i\ge-1}I_i) $ we have that \begin{equation*}
	\begin{split}
	\eps\le |\phi_x(t)| <1,
	\end{split}
	\end{equation*} from which it follows by 
\eqref{est_good_tan}	
	 that \begin{equation*}
	\begin{split}
	\left| \frac{u_{tan}(t,\phi_x(t))}{|\phi_x(t)|} - \frac{1}{2} \right| \le 
C \frac{\delta^{1/4}}{|\phi_x(t)|}\leq 
C\frac{\delta^{1/4}}{\eps}\leq C\delta^{1/12}.
	\end{split}
	\end{equation*} 
 	Therefore, we obtain the lower bound \begin{equation*}
	\begin{split}
	\int_{[0,T]\backslash (\cup_{i\ge-1}I_i) } \frac{u_{tan}(t,\phi_x(t))}{|\phi_x(t)|} dt \ge T(1-C\eps)\left(\frac{1}{2}-C\delta^{1/12}\right). 
	\end{split}
	\end{equation*} 
	Collecting all the estimates, we obtain\begin{equation*}
	\begin{split}
	\frac{1}{T}\int_0^T \frac{u_{tan}(t,\phi_x(t))}{|\phi_x(t)|} dt &\ge 
 (1-C\eps)\left(\frac{1}{2}-C\delta^{1/12}\right) - C{\sqrt{\eps}}
\ge \frac{1}{2}-C{\delta^{1/12}}.
	\end{split}
	\end{equation*}  {This gives the lower bound \begin{equation*}
	\begin{split}
	\frac{N_x(T)}{T}  \ge  \frac{1}{4\pi}- C{\delta^{1/12}},
	\end{split}
	\end{equation*} while the upper bound can be obtained similarly.} The proof for the first statement is now complete. To see the second statement, we just observe that with $T_n:=n$ \begin{equation*}
	\begin{split}
	H:=\limsup_{n} H_{T_n} = \bigcap_{m\ge 1} \bigcup_{n\ge m} H_{T_n} 
	\end{split}
	\end{equation*}  has measure greater than $|\Omega_0|-C{\delta^{1/12}}$. \end{proof}

\begin{remark}
	Since Yudovich's log-Lipschitz estimate in principle allows for arbitrarily fast rotation, it seems like an interesting problem to ask whether there is an initial data $\omega_0\in L^1\cap L^\infty$ which has a trajectory winding around the origin arbitrarily many times in a fixed time interval. This is related to the (difficult) question of whether instantaneous spiraling is possible for vortex patches. We refer the interested readers to discussions in \cite[Sec. 5--6]{EJSVP1}, \cite{EG, Pull3}. 
\end{remark}

{
\begin{remark} In a recent related work \cite{Choi2019}, the travel distance for fluid particles $$d_x(t):= \int_0^t |u(s,\phi_x(s))|ds$$ was considered, and it was established that for most particles on the initial vortex patch, this quantity grows linearly in time when the initial patch is  disk-like. 
\end{remark}}

\begin{remark}
	It is also interesting to ask a similar question about winding number (or travel distance) on   well-known solutions other than a disk, whose certain stability has been known. It might include   Kirchhoff's ellipse \cite{tang}, Lamb dipoles \cite{AC2019}, {Hill's vortex \cite{Choi2020}}, shear flows \cite{BesMas} and rectangles in a strip \cite{bd},   and so on.
\end{remark}
 
{\Large \section*{Acknowledgement}}

\noindent KC has been supported by   the National Research Foundation of Korea (NRF-2018R1D1A1B07043065) and by the Research Fund (1.200085.01) of UNIST(Ulsan National Institute of Science \& Technology).
IJ has been supported by the Science Fellowship of POSCO TJ Park Foundation and the National Research Foundation of Korea grant 2019R1F1A1058486. {IJ thanks Sun-Chul Kim for telling about the article \cite{MR}.}

\ \\ 


\begin{thebibliography}{10}

\bibitem{AC2019}
Ken Abe and Kyudong Choi.
\newblock Stability of {L}amb dipoles.
\newblock {\em preprint, arXiv:1911.01795}.

\bibitem{BaCh}
H.~Bahouri and J.-Y. Chemin.
\newblock \'{E}quations de transport relatives \'a\ des champs de vecteurs
  non-lipschitziens et m\'ecanique des fluides.
\newblock {\em Arch. Rational Mech. Anal.}, 127(2):159--181, 1994.

\bibitem{BesMas}
Jacob Bedrossian and Nader Masmoudi.
\newblock Inviscid damping and the asymptotic stability of planar shear flows
  in the 2{D} {E}uler equations.
\newblock {\em Publ. Math. Inst. Hautes \'{E}tudes Sci.}, 122:195--300, 2015.

\bibitem{bd}
Jennifer Beichman and Sergey Denisov.
\newblock 2{D} {E}uler equation on the strip: stability of a rectangular patch.
\newblock {\em Comm. Partial Differential Equations}, 42(1):100--120, 2017.

\bibitem{Choi2020}
Kyudong Choi.
\newblock Stability of {H}ill's spherical vortex.
\newblock {\em preprint, arXiv:2011.06808}.

\bibitem{Choi2019}
Kyudong Choi.
\newblock On the estimate of distance traveled by a particle in a disk-like
  vortex patch.
\newblock {\em Appl. Math. Lett.}, 97:67--72, 2019.

\bibitem{CJ}
Kyudong Choi and In-Jee Jeong.
\newblock Growth of perimeter for vortex patches in a bulk.
\newblock {\em Appl. Math. Lett.}, 113:106857, 9, 2021.

\bibitem{EJSVP1}
Tarek~M. Elgindi and In-Jee Jeong.
\newblock On singular vortex patches, {I}: Well-posedness issues.
\newblock {\em Memoirs of the AMS, to appear, arXiv:1903.00833}.

\bibitem{EJSVP2}
Tarek~M. Elgindi and In-Jee Jeong.
\newblock On singular vortex patches, {II}: long-time dynamics.
\newblock {\em Trans. Amer. Math. Soc.}, 373(9):6757--6775, 2020.

\bibitem{EG}
V.~Elling and M.~V. Gnann.
\newblock Variety of unsymmetric multibranched logarithmic vortex spirals.
\newblock {\em European J. Appl. Math.}, 30(1):23--38, 2019.

\bibitem{isg}
Drago\c{s} Iftimie, Thomas~C. Sideris, and Pascal Gamblin.
\newblock On the evolution of compactly supported planar vorticity.
\newblock {\em Comm. Partial Differential Equations}, 24(9-10):1709--1730,
  1999.

\bibitem{yud}
V.~I. Judovi\v{c}.
\newblock Non-stationary flows of an ideal incompressible fluid.
\newblock {\em \v{Z}. Vy\v{c}isl. Mat. i Mat. Fiz.}, 3:1032--1066, 1963.

\bibitem{MB}
Andrew~J. Majda and Andrea~L. Bertozzi.
\newblock {\em Vorticity and incompressible flow}, volume~27.
\newblock Cambridge University Press, Cambridge, 2002.

\bibitem{MaPu}
Carlo Marchioro and Mario Pulvirenti.
\newblock {\em Mathematical theory of incompressible nonviscous fluids},
  volume~96 of {\em Applied Mathematical Sciences}.
\newblock Springer-Verlag, New York, 1994.

\bibitem{MR}
T.~B. Mitchell and L.~F. Rossi.
\newblock The evolution of {K}irchhoff elliptic vortices.
\newblock {\em Physics of Fluids}, 20(5):054103, 2008.

\bibitem{Pull3}
D.~I. Pullin.
\newblock On similarity flows containing two-branched vortex sheets.
\newblock In {\em Mathematical aspects of vortex dynamics ({L}eesburg, {VA},
  1988)}, pages 97--106. SIAM, Philadelphia, PA, 1989.

\bibitem{sv}
Thomas~C. Sideris and Luis Vega.
\newblock Stability in {$L^1$} of circular vortex patches.
\newblock {\em Proc. Amer. Math. Soc.}, 137(12):4199--4202, 2009.

\bibitem{tang}
Yun Tang.
\newblock Nonlinear stability of vortex patches.
\newblock {\em Trans. Amer. Math. Soc.}, 304(2):617--638, 1987.

\bibitem{wp}
Y.~H. Wan and M.~Pulvirenti.
\newblock Nonlinear stability of circular vortex patches.
\newblock {\em Comm. Math. Phys.}, 99(3):435--450, 1985.

\end{thebibliography}

\end{document}